\documentclass{amsart}
\usepackage{tikz}

\usepackage[a4paper,margin=1.20in]{geometry}
\usetikzlibrary{cd}
\usepackage[utf8]{inputenc}
\usepackage[
    bookmarks=true,
    bookmarksnumbered=true,
    bookmarkstype=toc,hidelinks]{hyperref}

\usepackage{amsthm}
\usepackage{amsmath}
\usepackage{amssymb}
\usepackage{amsfonts}
\usepackage{mleftright}
\usepackage{mathrsfs}
\usepackage{physics}
\DeclareMathOperator{\Aut}{Aut}
\DeclareMathOperator{\Gal}{Gal}

\DeclareMathOperator{\End}{End}
\DeclareMathOperator{\Mod}{Mod}

\DeclareMathOperator{\id}{id}
\DeclareMathOperator{\GL}{GL}

\DeclareMathOperator{\ab}{ab}

\DeclareMathOperator{\Out}{Out}

\DeclareMathOperator{\rec}{rec}

\theoremstyle{definition}
\newtheorem{dfn}{Definition}[section]
\newtheorem*{dfn*}{Definition}
\newtheorem{rem}[dfn]{Remark}
\theoremstyle{plain}
\newtheorem{thm}[dfn]{Theorem}

\newtheorem{thma}{Theorem}

\newtheorem*{thm*}{Theorem}

\newtheorem*{prop*}{Proposition}
\newtheorem{cor}[dfn]{Corollary}
\newtheorem{lem}[dfn]{Lemma}
\newtheorem*{lem*}{Lemma}
\newtheorem*{rem*}{Remark}
\newtheorem*{ques*}{Question}

\newtheorem*{claim*}{Claim}
\newcommand{\Q}{\mathbb{Q}}
\newcommand{\Z}{\mathbb{Z}}

\title[Lubin-Tate characters and Aut-intrinsic Hodge-Tate-ness]{Lubin-Tate representations over nontrivial finite Galois extensions of $\Q_{p}$ are not Aut-intrinsically Hodge-Tate}
\author{Kaiji Kondo}
\date{}


\address{Research Institute for Mathematical Sciences, Kyoto University, Kyoto 606-8502, JAPAN}
\email{kkondo@kurims.kyoto-u.ac.jp}

\begin{document}

\begin{abstract}
    
In the present paper, we show that, for a prime number $p$ and a nontrivial finite Galois extension $k$ of $\mathbb{Q}_p$, under the additional assumption that $\sqrt{-1}\in k$ if $p=2$, the continuous $p$-adic representation of the absolute Galois group $G_k$ determined by a Lubin-Tate formal group over the ring of integers of $k$ has the following property: for almost all outer automorphisms of $G_k$, its composite with a representative of the outer automorphism is not Hodge-Tate.
More precisely, the set of outer automorphisms for which the resulting representation is Hodge-Tate is contained in a closed subgroup of $\operatorname{Out}(G_k)$ of Haar measure zero.
In particular, this representation is not Aut-intrinsically Hodge-Tate [in the sense of Hoshi].
In the case where $p$ is odd, this settles the odd-degree cases left open in the previous works of Hoshi and the author and, together with the known even-degree case, completes the picture for finite Galois extensions of $\mathbb{Q}_p$.
This exhibits a sharp contrast, from the viewpoint of anabelian geometry, between the $p$-adic cyclotomic character and other $p$-adic Lubin-Tate characters.
\end{abstract}

\maketitle

\tableofcontents

\section*{Introduction}
Let $p$ be a prime number, $k$ a finite extension of $\Q_{p}$, and $\overline{k}$ an algebraic closure of $k$.
We write $G_{k}\stackrel{\mathrm{def}}{=} \Gal(\overline{k}/k)$ for the absolute Galois group of $k$ determined by the algebraic closure $\overline{k}$.
In anabelian geometry, it is natural to discuss conditions for a continuous automorphism of $G_{k}$ to be induced by a field automorphism of $k$ [cf., e.g., \cite{8}, \cite{13}].
The following theorem studies such conditions from the perspective of $p$-adic Hodge theory:

\begin{thma}[\cite{8}, Corollary 3.4]
    Let $\alpha \colon G_{k} \stackrel{\sim}{\longrightarrow} G_{k}$ be a continuous automorphism of $G_{k}$.
Then the following conditions are equivalent:
\begin{enumerate}
    \item [(1)] The automorphism $\alpha$ is induced by a field automorphism of $k$.
    \item [(2)] For every finite-dimensional continuous $p$-adic representation $\rho \colon G_{k} \to \GL_{n}(\Q_{p})$ of $G_{k}$ that is Hodge-Tate, the composite $G_{k} \stackrel{\alpha}{\to} G_{k} \stackrel{\rho}{\to} \GL_{n}(\Q_{p})$ is Hodge-Tate.
    \item [(3)] The automorphism $\alpha$ is of HT-qLT-type [cf. \cite{8}, Definition 1.3, $\rm(\hspace{.08em}ii\hspace{.08em})$].
\end{enumerate}
 
\end{thma}

As suggested by this theorem and \cite{8}, Definition 1.3, $\rm(\hspace{.08em}ii\hspace{.08em})$, Lubin-Tate characters play an important role in the study of the \textit{geometricity} [i.e., the condition to be induced by a field automorphism] of continuous automorphisms of $G_{k}$.

Moreover, motivated by this theorem, Hoshi defined the following notion for continuous $p$-adic representations.

\begin{dfn*}[\cite{11}, Definition 1.3]
    Let $V$ be a finite-dimensional $\Q_{p}$-vector space and $\rho \colon G_{k} \to \Aut_{\Q_{p}}(V)$ a continuous $p$-adic representation.
    Then we shall say that $\rho$ is \textit{Aut-intrinsically Hodge-Tate} if, for an arbitrary continuous automorphism $\alpha$ of $G_{k}$, the composite $\rho \circ \alpha \colon G_{k} \to \Aut_{\Q_{p}}(V)$ is Hodge-Tate.
\end{dfn*}

Let $\pi \in \mathcal{O}_k$ be a uniformizer of the ring of integers $\mathcal{O}_k$ of $k$.
In the remainder of the present introduction, we write $\rho_{k,\pi} \colon G_{k} \to \Aut_{\Q_{p}}(k_{+})$ for the continuous $p$-adic representation obtained by forming the composite
        \begin{align*}
            G_{k} \stackrel{\chi_{k,\pi}}{\to} \mathcal{O}_{k}^{\times} \hookrightarrow \mathrm{Aut}_{\Q_{p}}(k_{+}),
        \end{align*}
        where the first arrow is the Lubin-Tate character $\chi_{k,\pi} \colon G_{k} \to \mathcal{O}_{k}^{\times}$ [i.e., the continuous character determined by a Lubin-Tate formal group law over $\mathcal{O}_{k}$ associated to $\pi$], and the second arrow is the natural inclusion.

It is natural to study which representations are Aut-intrinsically Hodge-Tate.
The following theorem is a typical example motivated by this question:

\begin{thma}[\cite{11}, Theorem 3.3; \cite{7}, Theorem 1.9; \cite{14}, Theorem 4.4]
    Let $V$ be a finite-dimensional $\Q_{p}$-vector space and $\rho \colon G_{k} \to \Aut_{\Q_{p}}(V)$ a continuous $p$-adic representation.
    Then the following assertions hold:
    \begin{enumerate}
        \item [(1)] Suppose either that $\rho$ is one-dimensional or that $\rho$ is two-dimensional and reducible.
        Then $\rho$ is Hodge-Tate if and only if $\rho$ is Aut-intrinsically Hodge-Tate.
        \item [(2)] Suppose that $p \geq 3$ and that $k/\Q_{p}$ is a finite Galois extension of even degree.
        Let $\pi \in \mathcal{O}_{k}$ be a uniformizer of $\mathcal{O}_{k}$.
        If $\rho$ is isomorphic to the continuous $p$-adic representation $\rho_{k,\pi}$ of $G_{k}$, then $\rho$ is not Aut-intrinsically Hodge-Tate.
        \item [(3)] Suppose that $p=2$ and that $k$ contains a primitive $4$-th root of unity.
        Suppose, moreover, that $k/\mathbb{Q}_p$ is an abelian extension.
        Let $\pi \in \mathcal{O}_{k}$ be a uniformizer of $\mathcal{O}_{k}$.
        If $\rho$ is isomorphic to the continuous $p$-adic representation $\rho_{k,\pi}$ of $G_{k}$, then $\rho$ is not Aut-intrinsically Hodge-Tate.
    \end{enumerate}
\end{thma}

\begin{rem*}
    We give some remarks on Theorem B.
    \begin{enumerate}
        \item [(1)] In the proof of Theorem B, (1), it is essential that the $p$-adic cyclotomic character can be reconstructed from $G_{k}$ in a group-theoretic manner.
        This is a significant difference between the $p$-adic cyclotomic character and other Lubin-Tate characters.
        \item [(2)] In \cite{11} and \cite{7} [resp. \cite{14}], Theorem B, (2) [resp. Theorem B, (3)] is not stated explicitly.
        However, the authors of \cite{11} and \cite{7} essentially established Theorem B, (2), while the author of \cite{14} essentially established Theorem B, (3).
        Theorem B, (2) was first established in \cite{11} for the case where $k/\Q_{p}$ is an abelian extension of even degree.
        In \cite{11}, by making use of the assumption that the extension is of even degree and abelian, the argument was reduced to the case where $k/\Q_{p}$ is a quadratic extension.
        \item [(3)] 
        In \cite{11} and \cite{7} [resp. \cite{14}], each of the authors constructed an explicit continuous automorphism $\varphi$ of $G_k$ such that $\rho_{k,\pi} \circ \varphi$ is not Hodge-Tate, in terms of the generators and relations established by Jannsen-Wingberg [cf. \cite{4}, Theorem 7.5.14] [resp. by Diekert [cf. \cite{23}, Theorem 3.1]].
    \end{enumerate}
\end{rem*}

In the present paper, we show the following theorem:

\begin{thma}[Theorem \ref{Haar measure zero}]
    Suppose that $k/\Q_p$ is a nontrivial finite Galois extension and that either $p \geq 3$ or $\sqrt{-1} \in k$.
    Let $\pi\in \mathcal{O}_k$ be a uniformizer.
    Endow $\Out(G_k)$ with its natural profinite topology and its [normalized] Haar measure.
    Then the subset 
    \begin{align*}
        E_{\rm out} \stackrel{\mathrm{def}}{=} \left\{ \alpha\in\Out(G_k) \ \middle|\ \rho_{k,\pi}\circ\widetilde{\alpha} \text{ is Hodge-Tate} \right\},
    \end{align*}
    where $\widetilde{\alpha}$ denotes an arbitrary representative of $\alpha$ in $\Aut(G_k)$, is contained in a closed subgroup of $\Out(G_k)$ of Haar measure zero.
\end{thma}

In other words, under the assumptions of Theorem C, a Lubin-Tate representation of $G_k$ fails to be Aut-intrinsically Hodge-Tate in a strong sense:
for almost all outer automorphisms $\alpha$ of $G_k$, the representation obtained by composing it with $\alpha$ is not Hodge-Tate.

In particular, we obtain the following theorem, which is a generalization of Theorem B, (2) and (3), above:

\begin{thma}[Corollary \ref{main theorem}]
    Suppose that $k/\Q_p$ is a nontrivial finite Galois extension and that either $p \geq 3$ or $\sqrt{-1} \in k$.
        Let $\pi \in \mathcal{O}_{k}$ be a uniformizer of $\mathcal{O}_{k}$.
        Then the continuous $p$-adic representation $\rho_{k,\pi} \colon G_k \to \Aut_{\Q_{p}}(k_{+})$ is not Aut-intrinsically Hodge-Tate.
\end{thma}

\begin{rem*}
We give some remarks on Theorem D.
\begin{enumerate}
    \item [(1)] In the case where $p \geq 3$, the novelty of Theorem D lies in the odd-degree cases [cf. Theorem B, (2)], although the argument proving the stronger Theorem C treats both odd and even degrees.
    From that point of view, the proof of Theorem C in the present paper is more uniform than the proofs of Theorem B, (2), in \cite{11} and \cite{7}.
    Theorem D also extends the previously known result by removing the assumption that $k/\Q_{2}$ is abelian [cf. Theorem B, (3)], under the assumption that $k/\Q_{2}$ is Galois and $\sqrt{-1} \in k$.
    \item [(2)] 
    In the proof of Theorem D, we do not give an explicit continuous automorphism $\varphi$ of $G_{k}$ such that $\rho_{k,\pi} \circ \varphi$ is not Hodge-Tate. 
    Thus, when $[k:\Q_p]$ is even, both the method based on explicit generators and relations employed in \cite{11}, \cite{7}, and \cite{14} [cf. Remark on Theorem B, (3)] and the method in the present paper have their own advantages and disadvantages.
\end{enumerate}

\end{rem*}
At the end of the Introduction, we describe the outline of the proof of Theorem C.
Let $\alpha$ be a continuous automorphism of $G_{k}$.
We write $\alpha_{+}$ for the automorphism of the $\mathbb{Q}_p$-vector space $k_{+}$ induced by the mono-anabelian reconstruction algorithms [cf. \cite{1}, Definition 3.10, (vi); \cite{1}, Proposition 3.11, (iv); \cite{2}, Lemma 1.2; see also the discussion entitled ``Groups of MLF-type'' in the section on notational conventions below].
Then the main new ingredient in the proof is a characterization of the Hodge-Tate-ness of $\rho_{k,\pi}\circ\alpha$ in terms of the induced automorphism $\alpha_{+}$.
More precisely, we show the following theorem:
\begin{thma}[Theorem \ref{key lemma}]
    Suppose that $k/\Q_{p}$ is a finite Galois extension.
    Let $\pi \in \mathcal{O}_{k}$ be a uniformizer of $\mathcal{O}_{k}$ and $\alpha$ a continuous automorphism of $G_{k}$.
    Then the continuous $p$-adic representation $\rho_{k,\pi}\circ\alpha$ is Hodge-Tate if and only if there exists an integer $i_\sigma$ for each $\sigma\in\operatorname{Gal}(k/\mathbb{Q}_{p})$ such that
\begin{align*}
    \alpha_{+} = \sum_{\sigma\in\operatorname{Gal}(k/\mathbb{Q}_p)} i_\sigma\cdot\sigma.
\end{align*}
In particular, if $\rho_{k,\pi}\circ\alpha$ is Hodge-Tate, then $\alpha_{+} \in \mathbb{Q}_p[\operatorname{Gal}(k/\mathbb{Q}_p)]^\times$.
\end{thma}

Thus, by means of the mono-anabelian reconstruction algorithms, the $p$-adic Hodge-theoretic condition that $\rho_{k,\pi}\circ\alpha$ be Hodge-Tate is translated into a linear-algebraic condition on the automorphism of $k_{+}$ induced by $\alpha$.
We then combine this characterization with the theory of mapping class groups and the theory of $p$-adic Lie groups to prove Theorem C.

\setcounter{section}{-1}

\section{Notational conventions}

\vspace{0.3cm}

\noindent {\bf Profinite groups}.
Let $G$ be a profinite group and $\alpha$ a continuous automorphism of the topological group $G$.
Then we shall write $G^{\ab}$ for the \textit{abelianization} of $G$ [i.e., the quotient of $G$ by the closure of the commutator subgroup of $G$] and $\alpha^{\mathrm{ab}}$ for the continuous automorphism of the topological group $G^{\mathrm{ab}}$ induced by $\alpha$ via the functoriality of abelianization.
\vspace{0.3cm}

\noindent {\bf Rings}. In the present paper, every ``ring'' is assumed to be unital and associative.
If $R$ is a ring, then we shall write $R_{+}$ for the underlying additive group of $R$ and $R^{\times} \subset R$  for the multiplicative group of units of $R$.
\vspace{0.3cm}

\noindent {\bf Mixed-characteristic local fields}. We shall refer to a field isomorphic to a finite extension of $\Q_{p}$, for some prime number $p$, as an \textit{MLF}. Here, ``MLF'' is to be understood as an abbreviation for ``mixed-characteristic local field''.
We recall that the valuation-theoretic structure of a finite extension of $\Q_{p}$ is functorially determined by its field structure.
Thus, for an MLF $k$, the usual valuation-theoretic objects are determined by the field structure of $k$.
Let $k$ be an MLF and $\overline{k}$ an algebraic closure of $k$.
Then we shall write
\begin{itemize}
    \item $\mathcal{O}_{k}$ for the ring of integers of $k$,
    \item  $\mathfrak{m}_{k} \subset \mathcal{O}_{k}$ for the maximal ideal of $\mathcal{O}_{k}$,
    \item $\underline{k} \stackrel{\mathrm{def}}{=}\mathcal{O}_{k}/\mathfrak{m}_{k}$ for the residue field of $\mathcal{O}_{k}$,
    \item $k^{(d=1)} \subset k$ for the [uniquely determined] minimal MLF contained in $k$,
     \item $p_{k}$ for the residue characteristic of $k$,
     \item $d_{k}\stackrel{\mathrm{def}}{=}[k \colon k^{(d=1)}]$ for the degree of the finite extension $k/k^{(d=1)}$,
     \item $f_{k}\stackrel{\mathrm{def}}{=}[\underline{k} \colon \underline{k}^{(d=1)}]$ for the degree of the finite extension $\underline{k}/\underline{k}^{(d=1)}$ [where we write $\underline{k}^{(d=1)}$ for the residue field of the ring of integers of the MLF $k^{(d=1)}$],
     \item $a_{k}$ for the largest nonnegative integer such that $k$ contains a primitive $p_{k}^{a_{k}}$-th root of unity,
    \item $G_{k} \stackrel{\mathrm{def}}{=} \Gal(\overline{k}/k)$ for the absolute Galois group of $k$ determined by the algebraic closure $\overline{k}$,
    \item  $I_{k} \subset G_{k}$ for the inertia subgroup of $G_{k}$, 
     \item $P_{k} \subset I_{k}$ for the wild inertia subgroup of $G_{k}$,
    \item  $\log_{k} \colon \mathcal{O}_{k}^{\times} \to k_{+}$ for the $p_{k}$-adic logarithm,
    \item $\widehat{k^{\times}}$ for the profinite completion of the multiplicative group $k^{\times}$ of $k$, and
    \item  $\rec_{k} \colon \widehat{k^{\times}} \stackrel{\sim}{\longrightarrow} G_{k}^{\ab}$ for the isomorphism induced by the reciprocity homomorphism $k^{\times} \hookrightarrow G_{k}^{\ab}$ in local class field theory.
\end{itemize}

\noindent {\bf Groups of MLF-type}.
We shall refer to a topological group isomorphic to the absolute Galois group of an MLF as a \textit{group of MLF-type}.

Let us recall [cf.\ \cite{1}, Definition 3.5; \cite{1}, Proposition 3.6; \cite{1}, Definition 3.10; \cite{1}, Proposition 3.11; \cite{25}, Definition 2.4, $\rm(\hspace{.08em}ii\hspace{.08em})$; \cite{25}, Proposition 2.5, $\rm(\hspace{.08em}i\hspace{.08em})$] that there exist functorial group-theoretic algorithms for constructing, from a group of MLF-type $G$,
\begin{itemize}
    \item a prime number $p(G)$,
    \item positive integers $d(G)$, $f(G)$,
    \item a nonnegative integer $a(G)$,
    \item subgroups $P(G) \subset I(G) \subset G$ of $G$,
    \item subgroups $\mathcal{O}^{\times}(G) \subset k^{\times}(G) \subset G^{\ab}$, whose final inclusion $k^{\times}(G) \subset G^{\ab}$ we denote by $\rec_{G}$, and
    \item a topological group $k_{+}(G)$,
\end{itemize}
 which ``correspond'' to
\begin{itemize}
    \item the prime number $p_{k}$,
    \item the positive integers $d_{k}$, $f_{k}$,
    \item the nonnegative integer $a_{k}$,
    \item the subgroups $P_{k} \subset I_{k} \subset G_{k}$ of $G_{k}$,
    \item the subgroups $\mathcal{O}_{k}^{\times} \subset k^{\times} \stackrel{\rec_{k}}{\hookrightarrow} G_{k}^{\ab}$, and
    \item the topological group $k_{+}$,
\end{itemize}
respectively.

By the functoriality of the algorithms, we have natural homomorphisms
\begin{align*}
    \Aut(G) \to \Aut(k^{\times}(G)),\ \ \ \Aut(G) \to \Aut(k_{+}(G)).
\end{align*}
 Since the automorphism of $G^{\ab}$ induced by an inner automorphism of $G$ is trivial, it follows from the constructions of $k^{\times}(G)$ and $k_{+}(G)$ that the automorphisms of $k^{\times}(G)$ and $k_{+}(G)$ induced by an inner automorphism of $G$ are trivial.
  Thus, the above two homomorphisms determine group homomorphisms
  \begin{align*}
      \Out(G) \to \Aut(k^{\times}(G)),\ \ \ \Out(G) \to \Aut(k_{+}(G)).
  \end{align*}
Let $\alpha$ be an element of $\Out(G)$. 
We write $\alpha^{\times}$ [resp. $\alpha_{+}$] for the image of $\alpha$ by this homomorphism $\Out(G) \to \Aut(k^{\times}(G))$ [resp. $\Out(G) \to \Aut(k_{+}(G))$].
In the present paper, we call $\alpha^{\times}$ and $\alpha_{+}$ \textit{the automorphisms induced from $\alpha$ by the mono-anabelian reconstruction algorithms}.

Let $k$ be an MLF, $\overline{k}$ an algebraic closure of $k$, and $\alpha$ an element of $\Out(G_{k})$.
By abuse of notation, we shall denote by $\alpha_{+} \colon k_{+} \stackrel{\sim}{\longrightarrow} k_{+}$, $\alpha^{\times} \colon k^{\times} \stackrel{\sim}{\longrightarrow} k^{\times}$ the respective images of $\alpha_{+} \colon k_{+}(G_{k}) \stackrel{\sim}{\longrightarrow} k_{+}(G_{k})$, $\alpha^{\times} \colon k^{\times}(G_{k}) \stackrel{\sim}{\longrightarrow} k^{\times}(G_{k})$ by the isomorphisms $\Aut(k_{+}(G_{k})) \stackrel{\sim}{\longrightarrow} \Aut(k_{+})$, $\Aut(k^{\times}(G_{k})) \stackrel{\sim}{\longrightarrow} \Aut(k^{\times})$ induced by the isomorphisms $k_{+} \stackrel{\sim}{\longrightarrow} k_{+}(G_{k})$, $k^{\times} \stackrel{\sim}{\longrightarrow} k^{\times}(G_{k})$ of \cite{1}, Proposition 3.11, $\rm(\hspace{.18em}i\hspace{.18em})$, $\rm(i\hspace{-.08em}v\hspace{-.06em})$.

\section{Representation induced by the mono-anabelian reconstruction algorithm}
In the present section, we recall, for the convenience of the reader, various notions introduced in \cite{7} and \cite{14}, as well as some relevant results.
Let $k$ be an MLF, $\overline{k}$ an algebraic closure of $k$, and $G$ a group of MLF-type.

\begin{rem}\label{a_k and primitive 4-th root of unity}
   It follows immediately from the various definitions involved that if $p_k=2$, then $a_k \geq 2$ if and only if $\sqrt{-1} \in k$.
   Thus, if $p(G)=2$, then the condition $a(G) \geq 2$ that appears frequently in the following discussion is precisely the group-theoretic condition corresponding to the existence of a primitive 4-th root of unity.
\end{rem}

\begin{thm}\label{generator and relation of group of MLF-type} 
   Write $\varpi$ for the unique element of $\widehat{\Z}=\prod_{\ell} \Z_{\ell}$ whose image in $\Z_{\ell}$ is given by $1$ if $\ell=p(G)$ [resp. by $0$ if $\ell \neq p(G)$].
   Suppose that either $p(G) \geq 3$ or $a(G) \geq 2$.
   Then there exist $\sigma$, $\tau$, $x_{0}, \ldots, x_{d(G)}\in G$, positive integers $s$, $t$, $u$, and an element $x_{1}^{\prime} \in \overline{\langle \sigma,\tau,x_{1}\rangle}$, where ``$\overline{\langle S \rangle}$'' denotes the closed subgroup of $G$ topologically generated by ``$S$'', such that the following conditions hold:
    \begin{enumerate}
    \item [(0)] We have
    \begin{align*}
        t\sum_{j=1}^{p(G)-1}u^{j}\neq p(G)-1.
    \end{align*}
     \item [(1)] The profinite group $G$ is presented as the profinite group topologically generated by $\sigma$, $\tau$, $x_{0}, \ldots, x_{d(G)}\in G$ and subject to the relations described in the conditions (2), (3), and (4) below.
  \item [(2)] The closed normal subgroup $P(G)$ of $G$ is pro-$p(G)$ and topologically normally generated by $x_{0},\ldots,x_{d(G)}$.
  \item [(3)] The elements $\sigma$, $\tau$ satisfy the relation $\sigma \tau \sigma^{-1}=\tau^{p(G)^{f(G)}}$.
  \item [(4)] The generators satisfy one further relation: 
   \begin{enumerate}
    \item for even $d(G)$,
    \begin{align*}
        \sigma x_{0} \sigma^{-1}=\Biggl(
\prod_{j=1}^{p(G)-1} x_{0}^{u^{p(G)-j}}\tau
\Biggr)^{\frac{\varpi t}{p(G)-1}} x_{1}^{p(G)^s}[x_{1},x_{2}][x_{3},x_{4}]\cdots[x_{d(G)-1},x_{d(G)}];
    \end{align*}
    \item for odd $d(G)$, 
    \begin{align*}
         \sigma x_{0} \sigma^{-1}=\Biggl(
\prod_{j=1}^{p(G)-1} x_{0}^{u^{p(G)-j}}\tau
\Biggr)^{\frac{\varpi t}{p(G)-1}}x_{1}^{p(G)^s}[x_{1},x_{1}^{\prime}][x_{2},x_{3}]\cdots[x_{d(G)-1},x_{d(G)}].
    \end{align*}
  \end{enumerate}
 
\end{enumerate}
    
\end{thm}

\begin{proof}
   When $p(G) \geq 3$, this assertion follows from \cite{4}, Theorem 7.5.14, together with \cite{1}, Proposition 3.6.
   When $p(G)=2$ and $a(G) \geq 2$, this assertion follows, in light of Remark \ref{a_k and primitive 4-th root of unity}, from \cite{23}, Theorem 3.1, together with \cite{1}, Proposition 3.6.
   See also \cite{2}, Proposition 1.1 [when $p(G) \geq 3$], and \cite{14}, Proposition 1.1 [when $p(G)=2$ and $a(G) \geq 2$].
\end{proof}

In the remainder of the present paper, we apply the notational conventions introduced in the statement of Theorem \ref{generator and relation of group of MLF-type} whenever either $p(G) \geq 3$ or $a(G) \geq 2$.
Moreover, if $p(G) \geq 3$ or $a(G) \geq 2$, then, for each $i=1,2,\ldots,d(G)$, write $y_{i} \in k_{+}(G)$ for the image of $x_{i}$ in $k_{+}(G)$ by the composite $P(G) \hookrightarrow I(G) \to \mathcal{O}^{\times}(G) \to k_{+}(G)$ [cf.\ conditions (1), (2) of Theorem \ref{generator and relation of group of MLF-type}; \cite{1}, Definition 3.10, $\rm(\hspace{.18em}i\hspace{.18em})$, $\rm(\hspace{.08em}ii\hspace{.08em})$, $\rm(\hspace{.06em}vi\hspace{.06em})$].

We recall that the topological group $k_{+}(G)$ has a natural structure of a $\Q_{p(G)}$-vector space of dimension $d(G)$ [cf. \cite{2}, Lemma 1.2; \cite{14}, Lemma 1.2] and that, for any continuous automorphism $\alpha$ of $G$, the induced automorphism $\alpha_{+}$ of $k_{+}(G)$ is an automorphism of $\Q_{p(G)}$-vector spaces.

\begin{lem}
    Suppose that $d(G)>1$ and that either $p(G) \geq 3$ or $a(G) \geq 2$.
    Then the $d(G)$ elements $y_{1},\ldots,y_{d(G)}$ defined in the discussion following Theorem \ref{generator and relation of group of MLF-type} form a basis of the $\Q_{p(G)}$-vector space $k_{+}(G)$.
\end{lem}

\begin{proof}
    When $p(G) \geq 3$, this assertion is none other than \cite{2}, Lemma 1.3.
    When $p(G)=2$ and $a(G) \geq 2$, this assertion is none other than \cite{14}, Lemma 1.3.
\end{proof}

One verifies easily that the isomorphism of topological groups $k_{+}(G_{k}) \stackrel{\sim}{\longrightarrow} k_{+}$ of \ \cite{1}, Proposition 3.11, $\rm(i\hspace{-.08em}v\hspace{-.06em})$, is also an isomorphism of $\Q_{p_{k}}$-vector spaces [where we recall that $p_{k} = p(G_{k})$].
By abuse of notation, if $d_{k}>1$ and either $p_{k} \geq 3$ or $a_{k} \geq 2$, then, for each integer $i$ satisfying $1 \leq i \leq d_{k}$, we write $y_{i} \in k_{+}$ for the image of $y_{i} \in k_{+}(G_{k})$ by the isomorphism $k_{+}(G_{k}) \stackrel{\sim}{\longrightarrow} k_{+}$ of \cite{1}, Proposition 3.11, $\rm(i\hspace{-.08em}v\hspace{-.06em})$.
\medskip

In the remainder of the present paper, if $d(G) > 1$ [resp. $d_{k} > 1$], and either $p(G) \geq 3$ [resp. $p_{k} \geq 3$] or $a(G) \geq 2$ [resp. $a_{k} \geq 2$], then we equip $k_{+}(G)$ [resp. $k_{+}$] with this basis, which allows us to identify $\Aut_{\Q_{p(G)}}(k_{+}(G))$ [resp. $\Aut_{\Q_{p_{k}}}(k_{+})$] with $\GL_{d(G)}(\Q_{p(G)})$ [resp. $\GL_{d_{k}}(\Q_{p_{k}})$].
We equip $\GL_{d(G)}(\Q_{p(G)})$ [resp. $\GL_{d_{k}}(\Q_{p_{k}})$] with the natural $p(G)$-adic [resp. $p_{k}$-adic] Lie group structure.
\medskip

\begin{dfn}\label{definition of absolutely Galois MLF}
    We shall say that the MLF $k$ is \textit{absolutely Galois} if the extension $k/k^{(d=1)}$ is a Galois extension. 
\end{dfn}

\begin{rem}\label{group ring}
    Since the elements of $\Gal(k/k^{(d=1)})$, regarded as a subset of $\End_{\Q_{p_k}}(k_+)$, are linearly independent over $\Q_{p_k}$ [cf. \cite{16}, Chapter $\rm VI$, \S 4, Theorem 4.1], the natural map
    \begin{align*}
        \Q_{p_{k}}[\Gal(k/k^{(d=1)})]\to \End_{\Q_{p_{k}}}(k_{+})
    \end{align*}
    is injective.
    In light of this observation, we shall regard $\Q_{p_{k}}[\Gal(k/k^{(d=1)})]$ as a sub-$\Q_{p_{k}}$-algebra of $\End_{\Q_{p_{k}}}(k_{+})$ via this injective homomorphism in the following discussion.

    In this context, we also observe that 
\begin{align*}
    \Aut_{\Q_{p_k}}(k_+) \cap \Q_{p_k}[\Gal(k/k^{(d=1)})]=\Q_{p_k}[\Gal(k/k^{(d=1)})]^{\times}.
\end{align*}
Indeed, the inclusion $\Aut_{\Q_{p_k}}(k_+) \cap \Q_{p_k}[\Gal(k/k^{(d=1)})] \supset \Q_{p_k}[\Gal(k/k^{(d=1)})]^{\times}$ is trivial.
Next, we prove the reverse inclusion.
Let $\beta \in \operatorname{Aut}_{\Q_{p_k}}(k_{+}) \cap \Q_{p_k}[\operatorname{Gal}(k/k^{(d=1)})]$ be an arbitrary element, and write
\begin{align*}
    X^n+a_1X^{n-1}+\cdots+a_{n-1}X+a_n
\end{align*}
for the minimal polynomial of $\beta$ over $\Q_{p_k}$.
Since $\beta \in \operatorname{Aut}_{\Q_{p_k}}(k_{+})$, we have $a_n \neq 0$.
Thus, we have
\begin{align*}
    \beta^{-1}= -a_{n}^{-1}(\beta^{n-1}+a_{1}\beta^{n-2}+\cdots+a_{n-1}) \in \Q_{p_k}[\Gal(k/k^{(d=1)})],
\end{align*}
which proves the desired inclusion.
\end{rem}

Next, we review the profinite group structure of $\Out(G)$.
It follows from \cite{4}, Theorem 7.4.1, and \cite{19}, Proposition 4.4.3, that $\Out(G)$ has a natural profinite group structure.
In the remainder of the present paper, we endow $\Out(G)$ with this profinite group structure.

\medskip

In the remainder of the present section, we assume that $d(G) >1$ [resp. $d_{k}>1$] and either $p(G) \geq 3$ [resp. $p_k \geq 3$] or $a(G) \geq 2$ [resp. $a_k \geq 2$].

\begin{lem}\label{closed map}
    The action $\mathrm{Out}(G) \curvearrowright k_{+}(G)$ which is defined via the mono-anabelian reconstruction algorithm is continuous.  
    In particular, the induced map 
    \begin{align*}
        \Phi \colon \Out(G) \to \Aut_{\Q_{p(G)}}(k_{+}(G)) \stackrel{\sim}{\longrightarrow} \mathrm{GL}_{d(G)}(\mathbb{Q}_{p(G)})
    \end{align*}
    is continuous and closed.
\end{lem}

\begin{proof}
    When $p(G) \geq 3$, this assertion is none other than \cite{7}, Lemma 2.13.
   When $p(G)=2$ and $a(G)\geq 2$, this assertion follows, in light of Theorem \ref{generator and relation of group of MLF-type} and the various definitions involved, by a similar argument to the argument applied in the proof of \cite{7}, Lemma 2.13.
\end{proof}

\begin{lem}\label{p-adic Lie group structure}
    The image of the homomorphism $\Phi$ of Lemma \ref{closed map} has a natural $p(G)$-adic Lie group structure.
\end{lem}

\begin{proof}
    This assertion follows immediately from Lemma \ref{closed map}, together with \cite{17}, Theorem 9.6.
\end{proof}

\begin{lem}\label{definition of Z}
    Suppose that $k$ is an absolutely Galois MLF.
    We shall write $Z$ for the image of $\Q_{p_{k}}[\mathrm{Gal}(k/k^{(d=1)})]^{\times}$ in $\mathrm{GL}_{d_{k}}(\Q_{p_{k}})$ via the isomorphism $\Aut_{\Q_{p_{k}}}(k_{+}) \stackrel{\sim}{\longrightarrow} \mathrm{GL}_{d_{k}}(\Q_{p_{k}})$ determined by the basis $y_{1}, \ldots, y_{d_{k}}$ of $k_{+}$.
    Then the following assertions hold:
    \begin{enumerate}
        \item [(1)] The subgroup $Z \subset \GL_{d_{k}}(\Q_{p_{k}})$ is closed and has a natural $p_{k}$-adic Lie group structure.
        \item [(2)] The inequality $\dim (Z) \leq d_{k}$ holds, where we write $\dim(Z)$ for the dimension of $Z$ as a $p_{k}$-adic Lie group.
    \end{enumerate}
\end{lem}

\begin{proof}
    Assertion (1) follows immediately from the [easily verified] fact that $\Q_{p_k}[\Gal(k/k^{(d=1)})]^{\times}\subset \Aut_{\Q_{p_k}}(k_+)$ is closed [cf. Remark \ref{group ring}], together with \cite{17}, Theorem 9.6.
    
    Assertion (2) follows immediately from Remark \ref{group ring} and the easily verified fact that
    \begin{align*}
        \dim( \Q_{p_k}[\Gal(k/k^{(d=1)})]) = d_k,
    \end{align*}
    where we write $\dim( \Q_{p_k}[\Gal(k/k^{(d=1)})])$ for the dimension of $\Q_{p_k}[\Gal(k/k^{(d=1)})]$ as a $p_{k}$-adic Lie group.
    This completes the proof of Lemma \ref{definition of Z}.
\end{proof}

\section{Classification of abelian Hodge-Tate representations}
In the present section, we review a classification theorem of abelian Hodge-Tate representations.
Moreover, by combining the classification of abelian Hodge-Tate representations with the mono-anabelian reconstruction algorithm, we obtain Theorem \ref{key lemma}, which will provide the key observation for the proof of the main theorem of the present paper.
Let $k$ be an MLF and $\overline{k}$ an algebraic closure of $k$.

In the present section, we do not make use of the presentation in terms of generators and relations of groups of MLF-type introduced in Theorem \ref{generator and relation of group of MLF-type}.
Thus, even when $p_k \geq 3$ or $a_k \geq 2$, we do not fix generators of $G_k$ satisfying the conditions of Theorem \ref{generator and relation of group of MLF-type}.

\begin{dfn}
    Suppose that $k$ is an absolutely Galois MLF [cf. Definition \ref{definition of absolutely Galois MLF}].
    Let $\pi \in \mathcal{O}_{k}$  be a uniformizer of $\mathcal{O}_{k}$ and $\sigma$ an element of $\Gal(k/k^{(d=1)})$.
    Then we shall write
    \begin{align*}
        \chi_{\pi,\sigma} \colon G_{k}^{\ab} \stackrel{\rec_{k}^{-1}}{\to} \widehat{k^{\times}} \twoheadrightarrow \mathcal{O}_{k}^{\times} \stackrel{\sigma}{\to} \mathcal{O}_{k}^{\times},
    \end{align*}
   where the second arrow is the projection determined by $\pi$.
\end{dfn}

\begin{thm}\label{classification of abelian Hodge-Tate rep}
    Suppose that $k$ is an absolutely Galois MLF.
    Let $\pi \in \mathcal{O}_{k}$ be a uniformizer of $\mathcal{O}_{k}$ and $\phi \colon G_{k}^{\mathrm{ab}} \to \mathcal{O}_{k}^{\times}$ a continuous homomorphism.
    Then the following two conditions are equivalent:
    \begin{enumerate}
        \item [(1)] The continuous $p_{k}$-adic representation obtained by forming the composite
        \begin{align*}
            G_{k} \twoheadrightarrow G_{k}^{\mathrm{ab}} \stackrel{\phi}{\to} \mathcal{O}_{k}^{\times} \hookrightarrow \mathrm{Aut}_{\Q_{p_{k}}}(k_{+})
        \end{align*}
        --- where the first arrow is the natural surjective continuous homomorphism, and the third arrow is the natural inclusion --- is Hodge-Tate.
        \item [(2)] There exist an integer $i_{\sigma}$ for each $\sigma \in \mathrm{Gal}(k/k^{(d=1)})$ and an open subgroup $J \subset I_{k}$ such that
        \begin{itemize}
            \item the restriction to $J$ of the composite of the natural surjective continuous homomorphism $G_{k} \twoheadrightarrow G_{k}^{\mathrm{ab}}$ and the given homomorphism $\phi \colon G_{k}^{\mathrm{ab}} \to \mathcal{O}_{k}^{\times}$
        \end{itemize}
        coincides with
        \begin{itemize}
        \item the restriction to $J$ of the composite of the natural surjective continuous homomorphism $G_{k} \twoheadrightarrow G_{k}^{\mathrm{ab}}$ and the homomorphism
            \begin{align*}
                \prod_{\sigma \in \mathrm{Gal}(k/k^{(d=1)})} \chi_{\pi, \sigma}^{i_{\sigma}} \colon G_{k}^{\mathrm{ab}} \to \mathcal{O}_{k}^{\times}.
            \end{align*}
        \end{itemize}
    \end{enumerate}
\end{thm}

\begin{proof}
   This assertion follows from the classification of abelian Hodge-Tate representations due to Serre [cf. \cite{12}, Chapter $\rm I\hspace{-.15em}I\hspace{-.15em}I$, \S A.5, Corollary].
   See \cite{11}, Lemma 1.8, for a formulation adapted to the present setting.
\end{proof}

\begin{dfn}
    Suppose that $k$ is an absolutely Galois MLF.
    Let $\pi \in \mathcal{O}_{k}$ be a uniformizer of $\mathcal{O}_{k}$.
    Then we write $\rho_{k,\pi} \colon G_{k} \to \mathrm{Aut}_{\Q_{p_{k}}}(k_{+})$ for the continuous $p_{k}$-adic representation of $G_{k}$ obtained by forming the composite
    \begin{align*}
        G_{k} \twoheadrightarrow G_{k}^{\mathrm{ab}} \stackrel{\chi_{\pi,\id_{k}}}{\to} \mathcal{O}_{k}^{\times} \hookrightarrow \mathrm{Aut}_{\Q_{p_{k}}}(k_{+}),
    \end{align*}
    where the first arrow is the natural surjective continuous homomorphism, and the third arrow is the natural inclusion.
\end{dfn}

\begin{rem}
    It follows from \cite{12}, Chapter $\rm I\hspace{-.15em}I\hspace{-.15em}I$, \S A.4, Proposition 4, that, up to taking the dual, the $p_{k}$-adic representation $\rho_{k,\pi}$ is isomorphic to the continuous $p_{k}$-adic representation determined by a Lubin-Tate character [i.e., a continuous character determined by a Lubin-Tate formal group over $\mathcal{O}_k$].
    Here, whether or not the dual occurs depends on the normalization of local reciprocity [equivalently, on the choice between arithmetic and geometric Frobenius].
    Thus, it follows from \cite{15}, Theorem 6.2, $(\rm i)$, that the continuous $p_{k}$-adic representation $\rho_{k,\pi}$ is crystalline [hence also Hodge-Tate].
\end{rem}

\begin{rem}\label{inertially independence}
   It follows from \cite{12}, Chapter $\rm I\hspace{-.15em}I\hspace{-.15em}I$, \S A.1, Corollary 2, that the Hodge-Tate-ness of continuous $p_{k}$-adic representations of $G_{k}$ is independent of the choice of representatives of an inertial equivalence class [cf., e.g., \cite{8}, Definition 1.2, $\rm(\hspace{.18em}i\hspace{.18em})$], where we recall that, for a finite-dimensional $\Q_{p_{k}}$-vector space $V$, two representations $\rho, \rho^{\prime} \colon G_{k} \to \Aut_{\Q_{p_{k}}}(V)$ are referred to as \textit{inertially equivalent} if there exists an open subgroup $J \subset I_{k}$ of $I_{k}$ such that $\rho|_{J} = \rho^{\prime}|_{J}$.
   Moreover, one verifies easily that the inertial equivalence class of $\chi_{\pi, \id_{k}}$ is independent of the choice of a uniformizer of $\mathcal{O}_{k}$.
   Thus, the choice of a uniformizer of $\mathcal{O}_{k}$ is inessential for the discussion that follows. 
\end{rem}

The following theorem constitutes the main novelty of the present paper:

\begin{thm}\label{key lemma}
    Suppose that $k$ is an absolutely Galois MLF.
    Let $\pi \in \mathcal{O}_{k}$ be a uniformizer of $\mathcal{O}_{k}$ and $\alpha$ a continuous automorphism of $G_{k}$.
    Then the following assertions are equivalent:
    \begin{enumerate}
        \item [(1)] The continuous $p_{k}$-adic representation $\rho_{k,\pi} \circ \alpha$ is Hodge-Tate.
        \item [(2)] There exists an integer $i_{\sigma}$ for each $\sigma \in \mathrm{Gal}(k/k^{(d=1)})$ such that
    \begin{align*}
        \alpha_{+}=\sum_{\sigma \in \mathrm{Gal}(k/k^{(d=1)})} i_{\sigma} \cdot \sigma.
    \end{align*}
    \end{enumerate}
    In particular, if $\rho_{k,\pi} \circ \alpha$ is Hodge-Tate, then $ \alpha_{+} \in \Q_{p_{k}}[\Gal(k/k^{(d=1)})]^{\times}$ [cf. Remark \ref{group ring}].
\end{thm}

\begin{proof}
    First, we show $(1) \Rightarrow (2)$.
    Suppose that $\rho_{k,\pi} \circ \alpha$ is Hodge-Tate.
    Then it follows from Theorem \ref{classification of abelian Hodge-Tate rep} that there exist an integer $i_{\sigma}$ for each $\sigma \in \mathrm{Gal}(k/k^{(d=1)})$ and an open subgroup $J \subset I_{k}$ such that
        \begin{itemize}
            \item the restriction to $J$ of the composite of the natural surjective continuous homomorphism $G_{k} \twoheadrightarrow G_{k}^{\mathrm{ab}}$ and the homomorphism 
            \begin{align*}
                \chi_{\pi,\id_{k}} \circ \alpha^{\mathrm{ab}} \colon G_{k}^{\ab} \to \mathcal{O}_{k}^{\times}
            \end{align*}
        \end{itemize}
        coincides with
                \begin{itemize}
        \item the restriction to $J$ of the composite of the natural surjective continuous homomorphism $G_{k} \twoheadrightarrow G_{k}^{\mathrm{ab}}$ and the homomorphism
            \begin{align*}
                \prod_{\sigma \in \mathrm{Gal}(k/k^{(d=1)})} \chi_{\pi, \sigma}^{i_{\sigma}} \colon G_{k}^{\mathrm{ab}} \to \mathcal{O}_{k}^{\times}.
            \end{align*}
        \end{itemize}
        Thus, it follows from the definition of $\alpha^{\times}$ that there exist an integer $i_{\sigma}$ for each $\sigma \in \mathrm{Gal}(k/k^{(d=1)})$ and an open subgroup $U \subset \mathcal{O}_{k}^{\times}$ such that \begin{itemize}
            \item the restriction to $U$ of the homomorphism 
            \begin{align*}
                \alpha^{\times} \colon \mathcal{O}_{k}^{\times} \to \mathcal{O}_{k}^{\times}
            \end{align*}
        \end{itemize}
        coincides with
                \begin{itemize}
        \item the restriction to $U$ of the homomorphism
            \begin{align*}
                \prod_{\sigma \in \mathrm{Gal}(k/k^{(d=1)})} \sigma_{i_{\sigma}} \colon \mathcal{O}_{k}^{\times} \to \mathcal{O}_{k}^{\times},
            \end{align*}
        \end{itemize}
        where we write $\sigma_{i_{\sigma}}$ for the endomorphism of $\mathcal{O}_{k}^{\times}$ defined by $x \mapsto \sigma(x)^{i_{\sigma}}$. 
        Let us recall that it follows from the definition of $\alpha_{+}$ and \cite{1}, Proposition 3.11, $\rm(i\hspace{-.08em}v\hspace{-.06em})$, that the following diagram commutes:
\[
\begin{tikzcd}
 k_{+} \arrow[r, "\alpha_{+}"] & k_{+} \\
   \mathcal{O}_{k}^{\times} \arrow[r, "\alpha^{\times}"'] \arrow[u, "\log_{k}"] &  \mathcal{O}_{k}^{\times} \arrow[u, "\log_{k}"'].
\end{tikzcd}
\]
Thus, it follows from \cite{13}, Lemma 4.1, together with the commutativity of the diagram above, that for each $\sigma \in \mathrm{Gal}(k/k^{(d=1)})$ there exists an integer $i_{\sigma}$ such that
\begin{align*}
    \alpha_{+}=\sum_{\sigma \in \mathrm{Gal}(k/k^{(d=1)})} i_{\sigma} \cdot \sigma.
\end{align*}
This completes the proof of the implication $(1) \Rightarrow (2)$.

Next, we show $(2) \Rightarrow (1)$.
It is well-known that the $p_{k}$-adic logarithm is injective on a sufficiently small open subgroup of $\mathcal{O}_{k}^\times$.
Thus, in light of the commutative diagram above, the assumption (2) implies that there exist an open subgroup $U \subset \mathcal{O}_{k}^{\times}$ and an integer $i_{\sigma}$ for each $\sigma \in \mathrm{Gal}(k/k^{(d=1)})$ such that
\begin{itemize}
            \item the restriction to $U$ of the homomorphism 
            \begin{align*}
                \alpha^{\times} \colon \mathcal{O}_{k}^{\times} \to \mathcal{O}_{k}^{\times}
            \end{align*}
        \end{itemize}
        coincides with
                \begin{itemize}
        \item the restriction to $U$ of the homomorphism
            \begin{align*}
                \prod_{\sigma \in \mathrm{Gal}(k/k^{(d=1)})} \sigma_{i_{\sigma}} \colon \mathcal{O}_{k}^{\times} \to \mathcal{O}_{k}^{\times},
            \end{align*}
        \end{itemize}
        where we write $\sigma_{i_{\sigma}}$ for the endomorphism of $\mathcal{O}_{k}^{\times}$ defined by $x \mapsto \sigma(x)^{i_{\sigma}}$. 
        By applying Theorem \ref{classification of abelian Hodge-Tate rep} in the reverse direction of the argument in the proof of $(1) \Rightarrow (2)$ and taking into account the definition of $\alpha^{\times}$, we obtain the desired assertion (1).
        This completes the proof of Theorem \ref{key lemma}.
\end{proof}

\begin{rem}
    The significance of this theorem lies in the fact that, by means of the mono-anabelian reconstruction algorithm, the Hodge-Tate-ness of $\rho_{k,\pi}\circ \alpha$ can be reduced to a linear-algebraic condition, which can then be studied by $p_{k}$-adic Lie-theoretic methods.
    If a similar result could be formulated for more general Hodge-Tate representations, it might provide a means of investigating whether such representations are Aut-intrinsically Hodge-Tate.
\end{rem}

\section{Mapping class groups and the proof of the main theorem}

In the present section, we prove the main theorem of the present paper by combining Theorem \ref{key lemma} with certain ``mapping class group and $p$-adic Lie group techniques'' developed in \cite{7}.

Let $k$ be an MLF, $\overline{k}$ an algebraic closure of $k$, and $G$ a group of MLF-type.
Let us recall that $\Out(G)$ [resp. $\Out(G_k)$] has been endowed with its natural profinite group structure.

We assume that $d(G) >1$ [resp. $d_{k}>1$] and either $p(G) \geq 3$ [resp. $p_k \geq 3$] or $a(G) \geq 2$ [resp. $a_k \geq 2$].
Let us recall that, in this situation, generators of $G$ [resp. $G_{k}$] satisfying Theorem \ref{generator and relation of group of MLF-type}, as well as the basis $y_1,\ldots,y_{d(G)}$ [resp. $y_1,\ldots,y_{d_{k}}$] of $k_{+}(G)$ [resp. $k_{+}$] determined by these generators, have been fixed.

First, we review the subgroup of $\mathrm{Out}(G)$ that corresponds to a ``mapping class group'' introduced in the discussion following \cite{7}, Lemma 2.15.

\begin{rem}
In \cite{7}, it is assumed that $p(G) \geq 3$ [resp. $p_{k} \geq 3$].
We note, however, that the mapping class group construction developed there depends only on the existence of a presentation in terms of generators and relations of ``Jannsen-Wingberg-Diekert type''.
Thus, it follows from Theorem \ref{generator and relation of group of MLF-type} that the same construction and arguments apply verbatim when $p(G)=2$ [resp. $p_k=2$] and $a(G)\geq 2$ [resp. $a_{k}\geq 2$].
\end{rem}

\begin{dfn}
    Suppose that $d(G) \geq 3$.
    \begin{enumerate}
        \item [(1)] We shall write
\begin{align*}
    g \stackrel{\mathrm{def}}{=} g(G) \stackrel{\mathrm{def}}{=} \begin{cases} \dfrac{d(G)-1}{2} & \text{if $d(G)$ is odd},\\[6pt] \dfrac{d(G)-2}{2} & \text{if $d(G)$ is even}. \end{cases}
\end{align*}
\item [(2)] Let $S$ be a closed orientable surface of genus $g\ (\geq 1)$ and $P$ a point on $S$.
Then we shall write $\mathrm{Mod}(S \setminus \{P\})$ for the mapping class group of $S \setminus \{P\}$.
\item [(3)] Let us regard
$\mathrm{Sp}_{2g}(\Z_{p(G)})$ as a subgroup of
$\GL_{d(G)}(\Q_{p(G)})$ via the injective group homomorphism
defined by
\begin{align*}
    A\longmapsto
\begin{cases}
\begin{bmatrix}
1 & 0_{1\times 2g}\\
0_{2g\times 1} & A
\end{bmatrix}
& \text{if $d(G)$ is odd},\\[16pt]
\begin{bmatrix}
I_2 & 0_{2\times 2g}\\
0_{2g\times 2} & A
\end{bmatrix}
& \text{if $d(G)$ is even}.
\end{cases}
\end{align*}

Also, recall that there exists [cf. the discussion following \cite{7}, Lemma 2.15] a map
\begin{align*}
  \rho \colon \Mod(S \setminus \{P\}) \to \Out(G),
\end{align*}
which is not necessarily a homomorphism of groups, such that the following diagram is commutative:
\[
\begin{tikzcd}
  \Out(G) \arrow[r, "\Phi"] & \GL_{d(G)}(\Q_{p(G)}) \\
  \Mod(S \setminus\{P\}) \arrow[r] \arrow[u, "\rho"] & \mathrm{Sp}_{2g}(\Z) \arrow[u, "\subset"'],
\end{tikzcd}
\]
where the lower horizontal arrow is the surjective homomorphism discussed in \cite{7}, Theorem 2.17.
We fix such a $\rho$.
Then we shall write $D \subset \Out(G)$ for the closed subgroup of $\Out(G)$ that is topologically generated by the image of $\rho$.
    \end{enumerate}
    
\end{dfn}

In the remainder of the present paper, whenever $d(G) \geq3$, we continue to regard $\mathrm{Sp}_{2g}(\Z_{p(G)})$ as a subgroup of $\GL_{d(G)}(\Q_{p(G)})$ via the above injective group homomorphism.
Also, in what follows, for a $p(G)$-adic [$p_{k}$-adic] Lie group $X$, we write $\dim(X)$ for the dimension of $X$ as a $p(G)$-adic [$p_{k}$-adic] Lie group.

\begin{lem}\label{dimention of the image of Phi}
    Suppose that $d(G) \geq 3$.
    Then the image of the homomorphism $\Phi \colon \Out(G) \to \mathrm{GL}_{d(G)}(\Q_{p(G)})$ of Lemma \ref{closed map} contains $\mathrm{Sp}_{2g}(\Z_{p(G)})$.
    In particular, we have $\dim (\mathrm{Im}(\Phi)) \geq 2g^{2}+g$.
\end{lem}

\begin{proof}
    It follows immediately from Lemma \ref{closed map}, the commutative diagram above, and the fact that the topological closure of $\mathrm{Sp}_{2g}(\Z)$ in $\GL_{d(G)}(\Q_{p(G)})$ is $\mathrm{Sp}_{2g}(\Z_{p(G)}) \subset \GL_{d(G)}(\Q_{p(G)})$ that $\Phi(D) \supset \mathrm{Sp}_{2g}(\Z_{p(G)})$.
    This completes the proof of the first assertion.
The second assertion follows immediately from the first assertion, together with the well-known equality $\dim(\mathrm{Sp}_{2g}(\Z_{p(G)})) = 2g^{2} + g$.
\end{proof}

\begin{lem}\label{Sp is not virtually abelian}
    Let $n$ be a positive integer.
    Then the profinite group $\mathrm{Sp}_{2n}(\Z_{p_{k}})$ is not a virtually abelian profinite group, i.e., has no abelian open subgroup.
\end{lem}

\begin{proof}
    When $p_{k} \geq 3$, this assertion is none other than \cite{7}, Lemma 2.15.
    Although $p_k \geq 3$ is assumed in \cite{7}, its proof applies verbatim to $p_k=2$.
    This completes the proof of Lemma \ref{Sp is not virtually abelian}.
\end{proof}

With the above preparations, we now prove the main theorem of the present paper.

\begin{thm}\label{Haar measure zero}
    Suppose that $k$ is an absolutely Galois MLF, that either $p_k \geq 3$ or $a_k \geq 2$, and that $d_k>1$.
    Let $\pi\in \mathcal{O}_{k}$ be a uniformizer of $\mathcal{O}_{k}$.
    We endow the profinite group $\Out(G_{k})$ with its [normalized] Haar measure.
    Then the subset
    \begin{align*}
        E_{\rm out} \stackrel{\mathrm{def}}{=} \left\{ \alpha\in\Out(G_k) \ \middle|\ \rho_{k,\pi}\circ\widetilde{\alpha} \text{ is Hodge-Tate} \right\},
    \end{align*}
    where $\widetilde{\alpha}$ denotes an arbitrary representative of $\alpha$ in $\Aut(G_k)$, is contained in a closed subgroup of $\Out(G_k)$ of Haar measure zero.
\end{thm}

\begin{proof}
    First, we note that the condition appearing in the definition of $E_{\rm out}$ is independent of the choice of the representative $\widetilde{\alpha}$.
    Indeed, precomposition by an inner automorphism changes a representation only by an isomorphism.

    Write $I\stackrel{\mathrm{def}}{=}\mathrm{Im}(\Phi)$, and $K\stackrel{\mathrm{def}}{=} I\cap Z$.
  Then it follows from Theorem \ref{key lemma} that if $\rho_{k,\pi}\circ\widetilde{\alpha}$ is Hodge-Tate, then $\Phi(\alpha) \in Z$.
  Consequently, we obtain
  \begin{align*}
      E_{\rm out}\subset \Phi^{-1}(K).
  \end{align*}
  Next, we prove the following claim:
  \begin{claim*}
      The closed subgroup $K \subset I$ is not open in $I$.
      Equivalently, the closed subgroup $K \subset I$ is not of finite index in $I$.
  \end{claim*}
  We first consider the case where $d_k\geq 5$.
  Let $g \stackrel{\mathrm{def}}{=} g(G_{k}) (\geq 2)$.
   Then it follows from Lemma \ref{dimention of the image of Phi} that $ \dim(I)\geq2g^2+g$.
   On the other hand, it follows from Lemma \ref{definition of Z}  that $\dim(K)\leq\dim(Z)\leq d_k$.
   If $d_k$ is odd, then $d_k=2g+1$, whereas if $d_k$ is even, then $d_k=2g+2$.
   Thus, the inequalities $d_k\leq2g+2<2g^2+g$ [cf. the condition that $g\geq2$].
   In particular, it follows that $\dim(K)<\dim(I)$.
   Thus, the closed subgroup $K \subset I$ cannot be open in $I$.
   
   We next consider the case where $d_{k}=3,4$ [i.e., $g=1$].
   In either case, since every group of order $3$ or $4$ is abelian, the group $\Gal(k/k^{(d=1)})$ is abelian.
   Hence the algebra $\Q_{p_k}[\Gal(k/k^{(d=1)})]$ is commutative, so $Z$, and therefore $K$, is abelian.
   On the other hand, it follows from Lemma \ref{dimention of the image of Phi} that $\mathrm{Sp}_2(\Z_{p_k})\subset I$.
   If $K$ were open in $I$, then $K\cap\mathrm{Sp}_2(\Z_{p_k})$ would be an abelian open subgroup of $\mathrm{Sp}_2(\Z_{p_k})$.
   This contradicts Lemma \ref{Sp is not virtually abelian}.
   Thus $K$ is not open in $I$ also in this case.
   
   It remains to consider the case where $d_k=2$.
   Write $\varphi$ for the automorphism of $G_k$ defined by
   \begin{align*}
       \varphi(\sigma)=\sigma,\qquad \varphi(\tau)=\tau,\qquad \varphi(x_2)=x_2x_1,\qquad \varphi(x_i)=x_i\quad(i\neq2).
   \end{align*}
   It follows from \cite{2}, Lemma 2.4 [when $p_{k} \geq 3$], and \cite{14}, Lemma 2.3 [when $p_{k}=2$ and $a_{k}\geq 2$], that, for every $n\in\Z\setminus\{0\}$, we have
   \begin{align*}
       (\varphi^n)_+ \bigl((k^{(d=1)})_+\bigr) \neq (k^{(d=1)})_+.
   \end{align*}
   On the other hand, since the group $\Gal(k/k^{(d=1)})$ acts trivially on $k^{(d=1)}$, and the restriction to $(k^{(d=1)})_+$ of an element of $\Q_{p_k}[\Gal(k/k^{(d=1)})]^\times$ is multiplication by its augmentation, which belongs to $\Q_{p_k}^\times$, every element of $Z$ preserves $(k^{(d=1)})_+$.
   Let $\overline{\varphi}$ denote the image of $\varphi$ in $\Out(G_k)$ and put $c\stackrel{\mathrm{def}}{=}\Phi(\overline{\varphi})\in I$.
   The preceding observations imply that $c^n\notin Z$ for all $n\in\Z\setminus\{0\}$.
   Suppose, for a contradiction, that $K$ is open in $I$, and write
   \begin{align*}
       C\stackrel{\mathrm{def}}{=}\overline{\langle c\rangle}\subset I.
   \end{align*}
   Then $C\cap K$ is an open subgroup of the profinite group $C$, hence is of finite index in $C$.
   In particular, it follows that there exists an integer $m>0$ such that
   \begin{align*}
       c^m\in C\cap K\subset Z,
   \end{align*}
   which contradicts the preceding assertion. 
   Thus, the closed subgroup $K \subset I$ is not open in $I$ also when $d_k=2$.
   We have therefore proved, in every case $d_k>1$, that the closed subgroup $K\subset I$ is not open [i.e., not of finite index].
   This completes the proof of the above claim.
   
   Since $I$ is compact [i.e., of finite Haar measure], this readily implies that $\mu_I(K)=0$, where $\mu_{I}$ is the normalized Haar measure of $I$.

   Since $\Phi \colon \Out(G_k)\twoheadrightarrow I$ is a continuous surjective homomorphism of compact groups, it follows from the uniqueness of normalized Haar measure that the pushforward of the normalized Haar measure $\mu_{\mathrm{out}}$ on $\Out(G_k)$ under $\Phi$ is the normalized Haar measure on $I$.
   Therefore, the equality
   \begin{align*}
       \mu_{\rm out}\bigl(\Phi^{-1}(K)\bigr)=\mu_I(K)=0
   \end{align*}
   hold.
   Since $E_{\mathrm{out}}\subset \Phi^{-1}(K)$, this completes the proof of Theorem \ref{Haar measure zero}.
\end{proof}

\begin{cor}\label{main theorem}
    Suppose that $k$ is an absolutely Galois MLF, that either $p_k \geq 3$ or $a_k \geq 2$, and that $d_k>1$.
    Let $\pi\in \mathcal{O}_{k}$ be a uniformizer of $\mathcal{O}_{k}$.
    Then the continuous $p_{k}$-adic representation $\rho_{k,\pi}$ is not Aut-intrinsically Hodge-Tate [cf. Definition in Introduction].
\end{cor}

\begin{proof}
   This assertion follows immediately from Theorem \ref{Haar measure zero}.
\end{proof}

\begin{rem}
    In \cite{7}, the author of the present paper proved the following assertion:
    \begin{quote}
        Suppose that $p_{k} \geq 3$, that $d_{k}$ is even, and that $k$ is an absolutely Galois MLF.
        Let $\pi\in \mathcal{O}_{k}$ be a uniformizer of $\mathcal{O}_{k}$.
        Let $\varphi$ be the automorphism of $G_{k}$ defined by the following equalities [cf.\ Theorem \ref{generator and relation of group of MLF-type}]:
    \begin{align*}
    \varphi(\sigma)=\sigma,\ \varphi(\tau)=\tau,\ \varphi(x_{2})=x_{2}x_{1},\ \varphi(x_{i})=x_{i}\ (i \neq 2).
    \end{align*}
    Then the continuous $p_{k}$-adic representation $\rho_{k,\pi} \circ \varphi$ is not Hodge-Tate.
    \end{quote}
    
    On the other hand, in general, we cannot obtain an explicit automorphism of $G_{k}$ that violates the Aut-intrinsic Hodge-Tate-ness of Lubin-Tate characters in the above proof of Corollary \ref{main theorem}.
\end{rem}

\begin{rem}
    Let $p$ be a prime number and $\overline{\Q}_{p}$ an algebraic closure of $\Q_{p}$.
    Then it is well-known that the Lubin-Tate character over $\Z_{p}$ determined by the uniformizer $p$ coincides with the $p$-adic cyclotomic character.
    Thus, the continuous $p$-adic representations $G_{\Q_{p}} \stackrel{\mathrm{def}}{=} \Gal(\overline{\Q}_{p}/\Q_{p}) \to \Q_{p}^{\times}$ determined by the Lubin-Tate characters over $\Z_{p}$ [i.e., with respect to arbitrary choices of uniformizers of $\Z_{p}$] are Aut-intrinsically Hodge-Tate [cf. Remark \ref{inertially independence}; \cite{13}, Proposition 1.1].
   This reflects a fundamental distinction between $\Q_p$ and its nontrivial finite Galois extensions from the point of view of anabelian geometry [at least] in the case where $p \geq 3$.
\end{rem}

\section*{Acknowledgments}

The author would like to express his sincere gratitude to Professor Yuichiro Hoshi for his numerous insightful discussions, warm encouragement, and detailed feedback on the drafts.
He is also grateful to Reiya Tachihara for carefully reading the drafts and providing invaluable comments.
He would also like to thank an anonymous referee for valuable comments and suggestions, which helped improve the formulation of the main theorem of the present paper.

\end{document}